\pdfoutput=1
\RequirePackage[l2tabu, orthodox]{nag}

\documentclass[12pt]{amsart}
\usepackage{fullpage,url,amssymb,enumerate,colonequals}
\usepackage[all]{xy} 
\usepackage{mathrsfs} 
\usepackage[dvipsnames]{xcolor}
\usepackage{tikz-cd}
\usepackage{breqn}
\usepackage[normalem]{ulem} 
\usepackage{textcomp}

\usepackage[OT2,T1]{fontenc}
\usepackage[utf8]{inputenc}
\DeclareSymbolFont{cyrletters}{OT2}{wncyr}{m}{n}
\DeclareMathSymbol{\Sha}{\mathalpha}{cyrletters}{"58}

\theoremstyle{definition}
\newtheorem{theorem}{Theorem}[section]
\newtheorem{lemma}[theorem]{Lemma}

\newtheorem*{theorem*}{Theorem}
\newtheorem*{corollary*}{Corollary}
\newtheorem*{assumption*}{Assumption}
\newtheorem{proposition}[theorem]{Proposition}

\theoremstyle{definition}
\newtheorem{definition}[theorem]{Definition}

\newtheorem{conjecture}[theorem]{Conjecture}
\newtheorem*{conjecture*}{Conjecture}

\newtheorem{example}[theorem]{Example}

\newtheorem{remark}[theorem]{Remark}
\newtheorem*{remark*}{Remark}

\makeatletter
\g@addto@macro\bfseries{\boldmath} 
\makeatother

\usepackage{microtype}

\usepackage[
]{hyperref}

\newcommand{\A}{\mathbb{A}}
\newcommand{\Cc}{\mathbb{C}}

\newcommand{\Pp}{\mathbb{P}}

\newcommand{\pp}{\mathbb{P}^2}
\newcommand{\ppp}{\mathbb{P}^3}
\newcommand{\Q}{\mathbb{Q}}
\newcommand{\R}{\mathbb{R}}

\newcommand{\Z}{\mathbb{Z}}

\newcommand{\Qbar}{{\overline{\Q}}}

\newcommand{\xg}{\mathfrak{g}}

\newcommand{\xp}{\mathfrak{p}}

\newcommand{\calP}{\mathcal{P}}

\DeclareMathOperator{\Bl}{Bl}

\DeclareMathOperator{\CH}{CH}

\DeclareMathOperator{\Ext}{Ext}

\DeclareMathOperator{\Gr}{Gr}

\DeclareMathOperator{\Jac}{Jac}

\DeclareMathOperator{\Pic}{Pic}

\DeclareMathOperator{\Nm}{Nm}

\usepackage{mathalfa}

\renewcommand{\H}{{\operatorname{H}}}

\newcommand{\dd}{\mathrm{d}}

\newcommand{\isom}{\simeq}

\usepackage{mathtools}

\newcommand{\norm}[1]{\lVert #1 \rVert}

\newcommand{\inv}{^{-1}}

\newcommand{\ii}{{\texttt i}}

\newcommand{\tpi}{2\pi{\texttt i}}
\newcommand{\tpip}{(2\pi{\texttt i})}

\DeclareMathOperator{\id}{id}

\DeclareMathOperator{\hgt}{ht}

\newcommand{\LMHS}{L}

\newcommand{\cc}{\mathcal{C}}
\newcommand{\cd}{\mathcal{D}}
\newcommand{\ce}{\mathcal{E}}

\newcommand{\co}{\mathcal{O}}

\newcommand{\cx}{\mathcal{X}}

\newcommand{\toi}{\hookrightarrow}
\newcommand{\isoto}{\overset{\sim}{\to}}

\newcommand{\coleq}{\colonequals}

\renewcommand{\Re}{\operatorname{Re}}
\renewcommand{\Im}{\operatorname{Im}}
\newcommand{\Ker}{\operatorname{Ker}}

\title{Computing heights via limits of Hodge structures}

\author[S.~Bloch]{Spencer Bloch}
\address{Spencer Bloch, University of Chicago, 5801 S Ellis Ave, Chicago, IL 60637, United States}
\email{bloch@math.uchicago.edu }

\author[R.~de~Jong]{Robin de Jong}
\address{Robin de Jong,  Universiteit Leiden, Niels Bohrweg 1, 2333 CA Leiden, The Netherlands}
\email{rdejong@math.leidenuniv.nl}

\author[E.~C.~Sert\"oz]{Emre Can Sert\"oz}
\address{Emre Can Sert\"oz, Institut f\"ur Algebraische Geometrie, Leibniz Universit\"at Hannover, 
Welfengarten 1, 30167 Hannover, Germany}
\email{emre@sertoz.com}

\date{\today}

\subjclass[2010]{\href{https://mathscinet.ams.org/msc/msc2010.html?t=11G50}{11G50},
\href{https://mathscinet.ams.org/msc/msc2010.html?t=14D06}{14D06},
\href{https://mathscinet.ams.org/msc/msc2010.html?t=14D07}{14D07},
\href{https://mathscinet.ams.org/msc/msc2010.html?t=14J20}{14J20},
\href{https://mathscinet.ams.org/msc/msc2010.html?t=14J70}{14J70},
\href{https://mathscinet.ams.org/msc/msc2010.html?t=14Q15}{14Q15},
\href{https://mathscinet.ams.org/msc/msc2010.html?t=14Q20}{14Q20}}


\begin{document}
\begin{abstract} 
We consider the problem of explicitly computing Beilinson--Bloch heights of homologically trivial cycles on varieties defined over number fields. Recent results have established a congruence, up to the rational span of logarithms of primes, between the height of certain limit mixed Hodge structures and certain Beilinson--Bloch heights obtained from odd-dimensional hypersurfaces with a node. This congruence suggests a new method to compute Beilinson--Bloch heights. Here we explain how to compute the relevant limit mixed Hodge structures in practice, then apply our computational method to a nodal quartic curve and a nodal cubic threefold. In both cases we explain the nature of the primes occurring in the congruence.
\end{abstract}

\maketitle

\section{Introduction}

Let $Y$ be a smooth projective variety of dimension $n=2p-1$ defined over a number field~$K$.
When $Z$ is an algebraic cycle on $Y$ of codimension~$p$ with trivial class in homology, we have the Beilinson--Bloch  height  $\hgt(Z) \in \R$ associated to $Z$, at least under standard assumptions~\cite{Bei1987, Bloch1984}, see Definition~\ref{def:bb_height}.

When $Y$ is a \emph{curve}, i.e., $n=1$, the cycle $Z$ is a degree-zero divisor on $Y$. The height $\hgt(Z)$ always exists and coincides with the non-normalized N\'eron--Tate height~\cite{Gross1986, Neron1965} of the point $[Z] \in \Jac(Y)$ determined by $Z$ on the Jacobian of~$Y$. The recent work~\cite{vBHM2020} by Van Bommel, Holmes and M\"uller gives an effective method for computing  $\hgt(Z)$ in this setting.

As far as we know, no general technique exists that is able to handle height computations in higher dimensions. In this article we will consider a setting which, in light of recent results \cite{Beil2022, BdJS2022}, becomes accessible to computations. One goal of this article is to introduce a computational technique to carry out the strategy (Section~\ref{sec:how_to_compute_lmhs}) and to demonstrate it on a three dimensional variety (Section~\ref{sec:higher_dim_example}). 

\subsection{A class of cycles obtained from nodal hypersurfaces} \label{sec:nodal_hyp}

Let $X_0 \subset \Pp^{n+1}$ be a geometrically integral projective hypersurface of dimension $n=2p-1$ defined over a number field $K$ and with a single ordinary double point. We assume $X_0$ has degree at least three to control its topology, see below. 

Let $Y$ denote the proper transform of $X_0$ under the blow-up of $\Pp^{n+1}$ at the node of $X_0$, and let $Q \subset Y$ be the exceptional quadric. This quadric is smooth and even-dimensional; we let $\Lambda_1, \Lambda_2$ denote the two rulings of $Q$. The degree constraint on $X_0$ forces that the difference $Z = \Lambda_1 - \Lambda_2$ of the two rulings is a homologically trivial cycle of codimension~$p$ on $Y$, see, e.g., \cite[Theorem~2.1]{Schoen1985}.

In this article, we will consider the Beilinson--Bloch height $\hgt(Z)$ of $Z$ on $Y$. 
Note that this real number exists under standard assumptions on $Z$ and $Y$, and could be viewed as a canonical height attached to the \emph{singular} variety $X_0$.
Recent work has related the height $\hgt(Z)$ to limiting periods associated to any smoothing deformation of~$X_0$ as we explain below. 

\subsection{Relation to limiting periods} \label{subsec:conj}

Consider a flat family of projective hypersurfaces $\pi \colon X \to S$, where $X$ is smooth, $S$ is a quasi-projective smooth curve over~$K$ with a base point $0 \in S(K)$, the family is smooth over $S \setminus \{0\}$, and the central fiber $\pi\inv(0)$ is identified with our nodal hypersurface $X_0$. We call the family $X/S$ a \emph{smoothing deformation} of $X_0$.

Let $\chi \in \Omega_{S,0} \simeq K$ be a non-zero cotangent vector. When $\sigma \colon K \toi \Cc$ is a complex embedding, we obtain an associated complex deformation $X_\sigma$ from $X$ by extending scalars. The cotangent vector $\chi$ and the complex deformation $X_\sigma$ together determine a limit mixed Hodge structure $\LMHS_{\chi,\sigma}$. It is a (twisted) biextension, as follows from Picard--Lefschetz theory, and it therefore has an associated biextension height $\hgt(\LMHS_{\chi,\sigma})$ in the sense of Hain \cite{Hain1990}. We define $\hgt(\LMHS_\chi) \coleq \sum_\sigma \hgt(\LMHS_{\chi,\sigma})$, where the sum ranges over all complex embeddings of~$K$. It is not difficult to show that the equivalence class $\hgt(\LMHS_\chi) \bmod \log |\Q^\times|$ is independent of the choice of $\chi$ and, in fact, of the deformation $X/S$.

With these notations, the first named author has recently made the following conjecture.
\begin{conjecture} \label{conj:bloch} The real numbers $\hgt(Z)$ and $\hgt(\LMHS_{\chi})$ differ by an element of $\Q \cdot \log |\Q^\times|$. 
\end{conjecture}
Let us refer to the difference $\hgt(Z)-\hgt(\LMHS_{\chi})$ in $\Q \cdot \log|\Q^\times|$ as the \emph{error term}.
Beilinson recently announced a proof of Conjecture~\ref{conj:bloch}, see~\cite{Beil2022}. His approach does not appear to give an explicit description of the error term in $\Q \cdot \log |\Q^\times|$. The error term is required to compute the height $\hgt(Z)$ explicitly using the limiting periods. We have given an explicit formula for the error term in dimension one~\cite{BdJS2022}, and a future work by the authors is aimed at describing the error term in all dimensions. 

\subsection{Aim of this paper}

In this article, we will demonstrate how to compute a numerical approximation of the real number $\hgt(\LMHS_{\chi})$ appearing in  Conjecture~\ref{conj:bloch}. 

After discussing an example dealing with a nodal plane curve, the case when $X_0$ is a nodal cubic \emph{threefold} will be the focus of this paper. In this setting, we prove that $\hgt(Z)$ is equal to the N\'eron--Tate height of a divisor on an associated sextic space curve. Combining this result with our method to compute $\hgt(\LMHS_{\chi})$ allows us to study the error term in Conjecture~\ref{conj:bloch}. In this case, we can give a geometric explanation for the primes supporting the error term.

Our result that allows us to reduce the determination of $\hgt(Z)$ to a height computation on a curve is as follows. 
Suppose that the nodal cubic threefold $X_0 \subset \Pp^4$ is such that the resolution $Y$ of $X_0$ has a proper regular model over the ring of integers of~$K$. In this case, the height $\hgt(Z)$ exists unconditionally. The moduli of lines contained in $X_0$ and passing through the node of $X_0$ is a smooth sextic curve $C \subset \ppp$. Let $\xg$ denote the difference of the two trigonal pencils on $C$. We denote by $\hgt(\xg)$ the non-normalized N\'eron--Tate height of the class of $\xg$ in the Jacobian of $C$.

\begin{theorem*}[Theorem~\ref{thm:threefold_to_curve}]
 With the set-up above, and in particular, assuming a regular model for our threefold, the Beilinson--Bloch height $\hgt(Z)$ and the non-normalized N\'eron--Tate height $\hgt(\xg)$ coincide, i.e., $\hgt(Z)=\hgt(\xg)$. 
\end{theorem*}

\subsection{Overview of this paper}

The limit mixed Hodge structures that appear in the context of this article are biextensions. We recall the relevant properties of biextensions in Section~\ref{sec:biextensions}.
The method of computing the periods of a biextension limit mixed Hodge structure is given in~Section~\ref{sec:how_to_compute_lmhs}.
We demonstrate our method of computing limit periods and our main theorem from~\cite{BdJS2022} on quartic plane curves in Section~\ref{sec:example_curves}.
In Section~\ref{sec:higher_dim_example} we demonstrate our method on a nodal cubic threefold. We apply Theorem~\ref{thm:threefold_to_curve} to determine and study the error term appearing in Conjecture~\ref{conj:bloch} in an example.
The proof of Theorem~\ref{thm:threefold_to_curve} appears in Section~\ref{sec:proof_thm_threefold_curve}. 

\subsection{Acknowledgments} 

We thank Vasily Golyshev, Matt Kerr, Greg Pearlstein, Matthias Sch\"utt, Duco van Straten, and the anonymous referees for helpful remarks.
We thank Raymond van Bommel, David Holmes and Steffen M\"uller for sharing with us their code related to \cite{vBHM2020} and for helpful discussions. 
We also acknowledge the use of Magma~\cite{Magma} and SageMath~\cite{sagemath} for facilitating experimentation.
The third author gratefully acknowledges support from MPIM Bonn. 

\section{Computing heights of biextensions} \label{sec:biextensions}

In this section we briefly recall the notion of a biextension, and explain how to compute the height of a biextension in practice. For a more detailed discussion we refer to \cite{BdJS2022} and \cite{Hain1990}.

For any integral mixed Hodge structure $V$ we will write $V_\Z \subset V_\R \subset V_\Cc$ for the underlying $\Z$-lattice and  real and complex vector spaces. The Hodge filtration on $V$ is denoted by $F^\bullet V \subset V_\Cc$ and the weight filtration by $W_\bullet V \subset V_\Z$. Note that  for biextensions we work with a stronger $\Z$-filtration instead of the usual $\Q$-filtration.

\begin{definition}\label{def:biextension} Let $H$ be a pure integral Hodge structure of weight~$-1$. A \emph{biextension} of $H$ is an integral mixed Hodge structure $B$ together with an identification of the weight graded pieces $\Gr^W B$ with $\Z(1)\oplus H \oplus \Z$.
\end{definition}
\begin{remark}
  Some authors do not consider the identification as part of the data. In that case, a choice of generators for $W_{-2} B \simeq \Z(1)$ and $W_{0}B/W_{-1}B \simeq \Z$ is called an \emph{orientation} for $B$. 
\end{remark}

Similarly, one has a notion of \emph{real} biextensions in the category of real (as opposed to integral) Hodge structures. In \cite{Hain1990} it is shown that the set of isomorphism classes of real biextensions of $H$ is canonically in bijection with the set $\R$ of real numbers. This identification allows one to attach a real number $\hgt(B)$ to any biextension $B$, by passing to the associated real biextension and taking its image in $\R$. The real number $\hgt(B)$ is called the \emph{height} of $B$ in the sense of Hain. 

We have shown in \cite{BdJS2022} that the height $\hgt(B)$ can be computed in terms of a \emph{period matrix} of $B$. We first recall how  this notion is defined. 
\begin{definition}\label{def:period_matrix}
 Let $B$ be a biextension on the  pure integral Hodge structure $H$.
Let $\gamma^0,\dots, \gamma^{2k+1}$ be a basis of $B_\Z$ that respects the weight filtration in the sense that 
\begin{itemize}
  \item $\gamma^0 = \tpi \in \Z(1)$,
  \item $\gamma^1, \, \dots, \, \gamma^{2k} \mod W_{-2}B$ is a basis for $H_\Z$, 
  \item $\gamma^{2k+1} \mod W_{-1}B = 1 \in \Z$.
\end{itemize}
Next, let $\omega_1,\dots,\omega_{k+1} \in F^0B$ be a basis that respects the weight filtration in the sense that
\begin{itemize}
  \item $\omega_1, \dots ,\omega_k \mod W_{-2}B \otimes \Cc$ is a basis for $F^0 H$,
  \item $\omega_{k+1} \mod W_{-1}B \otimes \Cc = 1 \in F^0 \Z$. 
\end{itemize}
 For each $i$, write $\omega_i = \sum_{j} a_{ij} \gamma^j$ with uniquely determined $a_{ij} \in \Cc$. The matrix $P_B = \left( a_{ij} \right)$ is called a \emph{period matrix} of the biextension $B$ (determined by the two chosen ordered bases). 
\end{definition}
Note that we can write a period matrix $P_B$ in block form as follows:
\begin{equation}\label{eq:PB}
  P_B = 
  \left( 
  \begin{array}{c|ccc|c}
       &  &     &  & \\
    b  &  & P_H &  & 0  \\
       &  &     &  & \\ \hline 
    c  &  & a   &  & 1
  \end{array}
 \right) \, , 
\end{equation}
where $P_H \in \Cc^{k \times 2k}$, $a \in \Cc^{1 \times 2k}$, $b \in \Cc^{k \times 1}$, and $c \in \Cc$. We observe that $P_H$ is a period matrix of the pure Hodge structure $H$. The vectors $a, b$ can be viewed as representing the extensions $W_0B/W_{-2}B \in \Ext^1(\Z,H)$ and $W_{-1}B \in \Ext^1(H,\Z(1))$.

\begin{definition}\label{def:hgt_of_matrix}  Denote the real and imaginary parts of the matrix $P_H$ by $\Re P_H$ and $\Im P_H$. 
  The \emph{height} of the period matrix $P_B$ is given by the expression 
  \begin{equation}\label{eq:hgt_PB}
  \hgt(P_B) = - 2 \pi \left(  \Im c - \Im a \cdot \begin{pmatrix} \Im P_H \\ \Re P_H \end{pmatrix}^{-1} \cdot 
    \begin{pmatrix} \Im b \\ \Re b \end{pmatrix}   \right).  
  \end{equation}
\end{definition}
We have the following result, see \cite[Theorem~2.9]{BdJS2022}.
\begin{theorem} \label{thm:explicit_height} Let $B$ be a biextension on the  pure integral Hodge structure $H$. Let $P_B$ be a period matrix of $B$. Then we have an equality $\hgt(B)=\hgt(P_B)$, where $\hgt(B)$ denotes the height of $B$ in the sense of Hain. 
\end{theorem}
Theorem~\ref{thm:explicit_height} allows to compute the height of a biextension in a straightforward manner once a period matrix is known.

\subsection{Twisted biextensions}
In geometrical contexts one often encounters Tate twists of biextensions. 

\begin{definition}\label{def:twisted_biextension}
  An integral mixed Hodge structure $V$ is called a \emph{$k$-twisted biextension} if its Tate twist $V(-k)$ is a biextension. The \emph{height} of $V$, notation $\hgt(V)$, is defined to be the height in the sense of Definition~\ref{def:hgt_of_matrix} of the biextension $V(-k)$. 
\end{definition}
\begin{remark} \label{rem:twisted}
The notion of a period matrix of a biextension naturally generalizes to the notion of a period matrix of a twisted biextension. Let $B$ be a biextension.
When $V=B(k)$ is a $k$-twisted biextension, and $P_V$ is a period matrix of $V$, then  $\tpip^{k} P_V$ is a period matrix of~$B$.
\end{remark}

\section{Computing limit mixed Hodge structures of families of hypersurfaces}\label{sec:how_to_compute_lmhs}

In~\cite{Ser2019}, the third named author describes an algorithm for translating the periods of one smooth hypersurface to another. Starting with the periods of the Fermat hypersurface, one can compute the periods of any other smooth hypersurface. Below, we describe an adaptation of this method to compute the limit periods of an odd dimensional nodal hypersurface. 

Although general limit mixed Hodge structures can be computed by the ideas presented here, the restriction to simple nodal degenerations significantly simplifies the exposition. Therefore, we will stick to this simple case, which suits our intended applications.

\subsection{The generic period matrix}

Let $n=2p-1$ with $p \in \Z_{> 0}$ an odd positive integer. We consider a family of $n$-dimensional hypersurfaces $X_t = V(f_t) \subset \Pp^{n+1}$, smooth for $t$ in a small open disk around $0$ in $\Cc$, and simply degenerating to a hypersurface $X_0 = V(f_0)$ with a single ordinary double point. Note that a simple degeneration is one where the total space is smooth, and hence, the central fiber appears with multiplicity $1$.

We will assume $f_t \in \Qbar(t)[x_0,\dots,x_{n+1}]$. In principle, the constraints on the base field can be relaxed. For example, in theory, we can work with $\Cc$ instead of $\Qbar$. However, computations require an effective subfield of $\Cc$. Also, we could assume $f_t$ depends on $t$ algebraically rather than rationally without losing computability. These relaxations do not significantly impact the discussion below.

Let $\mu_1(t),\dots,\mu_{2k+2}(t) \in \H^n(X_t,\Cc)$ be a basis, varying holomorphically in $t$, such that $\mu_1(t),\dots,\mu_{k+1}(t)$ are in  $F^{p} \H^n(X_t,\Z)$. 
We recall from~\cite{Griffiths1969vI} that we can represent each $\mu_i(t)$ as the residue of a meromorphic top form  on $\Pp^{n+1} \setminus X_t$ with polar locus $X_t$ and with the numerator depending polynomially on $t$. 

Let $\gamma_0(t),\dots,\gamma_{2k+1}(t) \in \H_n(X_t,\Z)$ be a locally constant basis near a regular value $t=t_0$ of the family $X_t$. Let $P(t) = \left( \int_{\gamma_j(t)} \mu_i(t) \right)$ denote the period matrix of $X_t$ with respect to these bases, and write $P_{ij}(t)$ for the $(i,j)$-th entry of $P(t)$. 

\subsection{The Gauss--Manin connection}
Using Griffiths--Dwork reduction applied to the meromorphic form on $\Pp^{n+1} \setminus X_t$ used to represent $\mu_i(t)$, we can compute differential operators $\cd_i \in \Qbar(t)[\frac{\partial}{\partial t}]$ for $i=1,\dots,k+1$ such that $\cd_i$ annihilates the $i$-th row of $P(t)$. 
We refer to~\cite{Dwork1962, Griffiths1969vI, Katz1968,  Lairez2015} for details; the exact procedure we use in our implementation is explained in~\cite[\S2.4-2.5]{Ser2019}. 

The next step is to compute the functions $P_{ij}(t) = \int_{\gamma_j(t)} \mu_i(t)$ from the $\cd_i$. It turns out that it already suffices  to determine the value $P(t_0)$ at a regular point $t_0 \in \Qbar$, as all the derivatives of $P_{ij}(t)$ at $t=t_0$ then easily follow by another set of Griffiths--Dwork reduction operations~\cite[p.~3003]{Ser2019}. Now the main result of~\cite{Ser2019} and the accompanying Magma~\cite{Magma} package \texttt{PeriodSuite}
allow us to compute such an initial value $P(t_0)$, which means approximating it to arbitrary precision with rigorous error bounds.

Each $\cd_i$ is a Picard--Fuchs differential equation~\cite{Deligne1970}. It follows that one can find a basis of solutions to each $\cd_i$ at $t=0$ where each basis element is a polynomial in $\log(t)$ with $\Qbar[[t]]$ coefficients. The SageMath~\cite{sagemath} package \texttt{ore\_algebra}~\cite{Kauers2013, Mezzarobba2016} is able to perform numerical holomorphic continuation by carrying a basis of solutions at $t=t_0$ to $t=0$ along a path on the complex plane whose interior avoids the singularities of $\cd_i$. 

This gives us the power series expansion of the period matrix,
\begin{equation}
  P(t) \in (\Cc[[t]][\log(t)])^{(k+1) \times 2(k+1)}.
\end{equation}
Our computations yield an approximation of the complex power series coefficients (with rigorous error bounds), and a truncation of the power series to some finite power. 

\subsection{Computing the limit mixed Hodge structure}

Recall that we are assuming a simple degeneration at $t=0$ to a hypersurface with a single ordinary double point. Picard--Lefschetz theory tells us that in fact $P(t) = P_0(t) + P_1(t)\log(t)$ with $P_0, P_1 \in  \Cc[[t]]$. The matrix $P_1(t)$ determines  the monodromy around $t=0$. 
 Although $P_1(0)$ is only known approximately, it is necessarily approximating a matrix with \emph{integer} coefficients. Therefore rounding its entries to the nearest integer gives us the matrix representation of the monodromy operator $T$ in terms of the homology basis $\gamma_i$. This computation is rigorous when the error of the entries is less than $0.5$.

From $P_1(0)$ we immediately obtain the weight filtration of the limit mixed Hodge structure on $\H^n(X_t,\Z)$ by computing $W_{n-1} = \Im (T-1)$, $W_n = \Ker (T-1)$, $W_{n+1} = \H^n(X_t,\Z)$.   From this point onward, possibly upon changing the homology basis appropriately, we assume that our $\gamma_i(t)$s actually respect the weight filtration. In our situation this means that $P_1(0)$ will have only a single corner entry equal to $\pm 1$ and $0$s everywhere else. 
We choose our basis so that the corner entry of $P_1(0)$ is $-1$. This choice of sign is a consequence of the Picard--Lefschetz formula if one insists that $\gamma^0\cdot \gamma^{2k+1} = 1$. 

The matrix $P_0(0)$ will be the period matrix of the limit mixed Hodge structure at $t=0$, with columns respecting the weight filtration. Our limit mixed Hodge structure will be a $(-p)$-twisted biextension in the sense of Definition~\ref{def:twisted_biextension}. The orientation of the underlying biextension is not intrinsic, but we choose it to be compatible with $P_1(0)$.  

 It follows  from the formula in Definition~\ref{def:hgt_of_matrix} that the height of a biextension is independent of a simultaneous change of signs in the generators of a biextension. As the matrix $P_1(0)$  fixes generators for $\Gr^W_{n-1} $ and $ \Gr^W_{n+1}$ up to a \emph{common} sign change, we see that the height of our limit mixed Hodge structure is well defined. 
 
By Theorem~\ref{thm:explicit_height} and Remark~\ref{rem:twisted}, computing the height of the period matrix $\tpip^{-p} P_0(0)$ gives us the height of the limit mixed Hodge structure associated to the family $X_t$ at $t=0$.

\begin{remark}  
  A detailed account of the limit mixed Hodge structure in the case where $X_t$ is a family of \emph{curves} is given in~\cite[Chapter~1]{Carlson2017}.
\end{remark}

\begin{remark}\label{rem:half_matrix} 
  Note that we are only using one half of a basis for cohomology for our computations, namely $\mu_1(t),\dots,\mu_{k+1}(t)$. The complex conjugate of this basis completes the basis. It is therefore no mystery that we can recover the $2(k+1) \times 2(k+1)$ integral monodromy operator $T$ from the $(k+1)\times 2(k+1)$ integral matrix $P_1(0)$. In practice, it is often  easier to observe that the row space of $P_1(0)$ is identified with $W_{n-1}$, and $W_n$ is the orthogonal space to $W_{n-1}$ with respect to the intersection product on homology.
\end{remark}

\section{Computing the height on nodal curves} \label{sec:example_curves}

Our main result in~\cite{BdJS2022} is a formula expressing the N\'eron--Tate height of a divisor $p-q$ on a curve $C$ in terms of a limiting period associated to a smoothing of the nodal curve $C/(p \sim q)$. In this section, we recall the precise statement and then apply the method in Section~\ref{sec:how_to_compute_lmhs} to compute heights on curves.

This exercise serves two purposes. First, we illustrate the general principle of how a refinement of Conjecture~\ref{conj:bloch} can be used to compute Beilinson--Bloch heights using limiting periods. Second, we do this in a setting where our calculation, and hence our implementation of the limit period computation, can be verified by the methods of~\cite{vBHM2020, Holmes2012, Muller2014}.

\subsection{Limit periods and heights for curves}

In this subsection, we recall the set-up and the main result of~\cite{BdJS2022}. Let $K$ be a number field. Let $X_0$ be a geometrically integral projective curve over $K$ containing a single node. 
Let $C$ be the normalization of $X_0$ and let $p, q$ be the two points in the preimage of the node of $X_0$ in $C$. 
After a quadratic base change, we may assume $p,q$ are defined over $K$. 
We are interested in the height $\hgt(Z)$ of the degree-zero divisor $Z=p-q$ on the curve $C$. We recall that  $\hgt(Z)$ equals the non-normalized N\'eron--Tate height of the class determined by the divisor $Z$ in the Jacobian of $C$. 

Let $\co_K$ be the ring of integers of $K$ and let $\cc$ be a proper regular model of $C$ over $\co_K$. Let $\overline{p}, \overline{q}$ denote the Zariski closures on $\cc$ of $p,q \in C(K)$. Let $S$ be a quasi-projective smooth curve over~$K$ with a base point $0 \in S(K)$ and let $X \to S$ be a smoothing deformation of $X_0$ over $S$. 

The Kodaira--Spencer map yields an identification $\Omega_{S,0} \isoto \Omega_{C,p} \otimes \Omega_{C,q}$ of $K$-vector spaces. The latter space carries an integral structure coming from the regular model $\cc$, namely the $\co_K$-lattice 
$\Omega_{\cc/\co_K,\overline{p}} \otimes_{\co_K} \Omega_{\cc/\co_K, \overline{q}} \subset \Omega_{C,p} \otimes_K \Omega_{C,q}$. By transport of structure we obtain an $\co_K$-lattice in $\Omega_{S,0}$ determined by $\cc$. 
This integral structure allows us to define a norm $\norm{\chi} \in \R_{\ge 0}$ for any $\chi \in \Omega_{S,0}$. 

Let $\Phi$ be a $\Q$-divisor on $\cc$ supported only on the closed fibers of $\cc$ over $\co_K$ such that $\overline{p}-\overline{q}+\Phi$ is of degree zero on every component of every fiber of $\cc$ over $\co_K$. It is not hard to see using elementary intersection theory on the regular arithmetic surface $\cc$ that such a $\Phi$ always exists.

For any two $\Q$-divisors $\cd,\ce$ on $\cc$ whose supports are disjoint over $C$ 
we define
\begin{equation}\label{eq:finite_intersection}
(\cd, \ce)_{\mathrm{fin}} \coleq \sum_\xp \iota_\xp(\cd,\ce) \log \Nm(\xp),
\end{equation}
where the sum runs over all maximal ideals of $\co_K$, where $\iota_\xp$ refers to the intersection multiplicity of the two divisors on $\cc$ over $\xp$, and where $\Nm(\xp)$ denotes the norm of $\xp$. The sum is indeed finite, and can be understood as the finite part of the Arakelov intersection of $\cd$ and $\ce$.
 
\begin{theorem}[Main Theorem of \cite{BdJS2022}] \label{thm:conj_curves}  Let $\chi \in \Omega_{S,0}$ be a non-zero cotangent vector. Let $\hgt(\LMHS_\chi)$ be the biextension height of the limit mixed Hodge structure $\LMHS_\chi$ as defined in Section~\ref{subsec:conj}. There is an equality of real numbers
  \begin{equation} \label{eq:conj_curves}
 \hgt(Z)  = \hgt(\LMHS_\chi) +
\log \norm{\chi} + 2 \,(\overline{p} \cdot \overline{q})_{\mathrm{fin}} - \left( (\overline{p}-\overline{q})\cdot \Phi \right)_{\mathrm{fin}}, 
  \end{equation}
with $\hgt(Z)$ the non-normalized N\'eron--Tate height of the divisor $Z=p-q$ on the curve $C$. 
\end{theorem}
We observe from~\eqref{eq:finite_intersection} that indeed the error term $\hgt(Z)-\hgt(\LMHS_\chi)$ lies in $\Q \cdot \log |\Q^\times|$. 
We also observe that Theorem~\ref{thm:conj_curves} gives a new method to compute N\'eron--Tate heights on curves. 
The aim of the remainder of this section is to present and discuss an explicit example of such a computation. 

\subsection{Nodal plane quartics}

For our example, we will consider a nodal plane quartic, that is, a nodal hypersurface $X_0 \subset \pp$ of degree $4$. 
The method in Section~\ref{sec:how_to_compute_lmhs} is applicable in this setting as well as Theorem~\ref{thm:conj_curves}. First, we set-up a deformation of $X_0$ and then we briefly describe how to compute the non-Archimedean terms in Equation~\ref{eq:conj_curves}. 

\subsubsection{Set-up and notation}

Let $f \in \Z[x,y,z]$ be a primitive, homogeneous polynomial of degree four. Suppose that $X_0=V(f) \subset \pp_\Q$ is geometrically irreducible with a single node at the origin $[0:0:1]$ with the two tangents at the origin defined over $\Q$. 

Take $S=\A^1$ with parameter $t$ and consider the family of quartics $X/S$ cut out by $f + t \cdot z^4$ for $t \in \A^1$. We note that the fibers $X_t$ are generically smooth and degenerate simply to the nodal curve $X_0$ at the origin. 

We write $\chi$ for the element $\mathrm{d} t|_0 \in \Omega_{S,0}$. We first calculate the height of the limit mixed Hodge structure $\LMHS_\chi$ determined by the family $X_t$ and our choice for $\chi $ following the method described in Section~\ref{sec:how_to_compute_lmhs}. 

\subsubsection{The non-Archimedean terms}

Besides the height of the limit mixed Hodge structure, there are three terms in Equation~\ref{eq:conj_curves} which are of non-Archimedean origin. We will describe here how to compute these terms.

Let $C$ denote the normalization of $X_0$ with two designated points $p, q$ over the node. As before, we write $Z=p-q$. The curve $C$ is a smooth genus two curve over $\Q$. 

With $f$ defined over $\Z$ and primitive, we can consider the naive plane model $\cx_0 \subset \pp_\Z$ of $X_0$ defined by $f$.
Blow-up the Zariski closure in $\cx_0$ of the node  of $X_0$ to get an arithmetic surface $\cc'$. The generic fiber of $\cc'$ is $C$. Let $\cc$ be any proper regular model resolving the singularities of $\cc'$. 

Computing the term $\Phi$ on $\cc$ is straightforward linear algebra once we have the intersection matrix of the components of all singular fibers of $\cc$.  The Magma package \emph{Regular models of arithmetic surfaces} by Donnelly computes regular models of smooth (weighted) plane curves over number fields; this includes the aforementioned intersection matrices as well as a method to determine the components to which a given rational point of $C$ specializes. The value of $\left( (\overline{p}-\overline{q})\cdot \Phi \right)_{\mathrm{fin}}$ can thus be computed.

For the term $(\overline{p},\overline{q})_{\mathrm{fin}}$, one could also use the general purpose code of~\cite[Section~2]{vBHM2020}. The integral structure on $\Omega_{C,p}\otimes \Omega_{C,q}$ can be computed using the definition in a straightforward fashion, although it is cumbersome. Alternatively, both $\norm{\chi}$ and $(\overline{p},\overline{q})_{\mathrm{fin}}$ can be worked out in our setting by computing formally locally around the node to simplify calculations.

\subsubsection{An explicit quartic} 

We demonstrate the procedure above with a specific quartic $X_0$ and the deforming family $X_t$:
\begin{equation}
(198x^2 + 325xy + 75y^2)z^2 + (x^2y + y^3)z + x^4 + {\color{red} t} z^4.
\end{equation}
Running the method of Section~\ref{sec:how_to_compute_lmhs} on $X_t$, we found that the height $\hgt(\LMHS_\chi)$ of the associated limit mixed Hodge structure $\LMHS_\chi$ equals
\begin{equation}
-12.220335776489136661293609468838149530311651887679\dots
\end{equation}
We note that the quadratic part of the equation for $X_0$ and its discriminant factor as follows:
\begin{equation}
198x^2 + 325xy + 75y^2 = (11x + 15y)(18x + 5y), \quad \Delta = (5\cdot 43)^2.
\end{equation}
By a local calculation on our regular model $\cc$ we find that
\begin{equation}
\log \norm{\chi} =  4 \log(5\cdot 43), \quad
2 \,(\overline{p} \cdot \overline{q})_{\mathrm{fin}} = 2\log(5\cdot 43) \, . 
\end{equation}
Finally the closed fibers of $\cc$ are all irreducible so that we may take $\Phi=0$ (start with the naive model in $\Pp(1,3,1)$ attained by blow-up and use Donnelly's work). Theorem~\ref{thm:conj_curves} implies that
\begin{equation} \label{eqn:final_ht} \begin{split}
  \hgt(Z) &  = \hgt(\LMHS_\chi) + 6 \log(5 \cdot 43) \\
   & = 20.003492392276840127148001610594058520196773347220\dots 
   \end{split}
\end{equation}

Our calculations were performed using a combination of Magma~\cite{Magma} and SageMath~\cite{sagemath}. 
The period calculations took a few minutes with over 200 digits of precision. 

\subsubsection{Verifying the height computation}

The normalization $C$ of $X_0$ can be given in $\Pp(1,3,1)$ by the equation
\begin{equation} \label{eqn:curve}
2X^5Z - 4X^4Z^2 + X^3Y + 7X^3Z^3 + X^2YZ + 68X^2Z^4 + 179XZ^5 + Y^2 - 53Z^6.
\end{equation}
The two points $p,q \in C \subset \Pp(1,3,1)$ that lie over the node of $X_0$ are
\begin{equation} \label{eqn:points}
   (4 : -2055 : 15) \, , \quad (13 : 2465 : -5).
\end{equation}
The N\'eron--Tate height $\hgt(Z)$ of the divisor $Z=p-q$ on $C$ can alternatively be calculated based on \eqref{eqn:curve} and \eqref{eqn:points} by the built-in Magma package \emph{Computation of canonical heights using arithmetic intersection theory} due to M\"uller with contributions by Holmes and Stoll. We checked the outcome in \eqref{eqn:final_ht} against the output of this built-in package. The results agree to numerical precision. The computations here took a few seconds.

\subsubsection{Remarks about performance}

It is clear that for the calculation of N\'eron--Tate heights on genus two curves, the built-in Magma package is to be preferred for the economy of time. Indeed, experimental genus 2 height computation within the framework of the Birch--Swinnerton-Dyer conjecture dates back over two decades~\cite{Flynn01} and had time to mature. 

We view our calculations with genus 2 curves merely as a proof of concept for computing heights via limit mixed Hodge structures. The asymptotic computational complexity of period computation by deformation has not been fully analyzed, although comparative performance for very high precision appears to be excellent~\cite[\S 5]{Ser2019}. The method presented here---and, in particular, its implementation---is only a few years old. We expect significant improvements in performance over the years. 


\section{Computing the height on nodal threefolds} \label{sec:higher_dim_example}

In this section we will consider Conjecture~\ref{conj:bloch} in the setting of a nodal cubic threefold $X_0 \subset \Pp^4$. Using the notation there, we would like to study the ``error term'' $\hgt(Z) - \hgt(\LMHS_{\chi})$ for a smoothing deformation of $X_0$. In particular, we would like to show with an example that the primes supporting this error term have a simple geometric explanation, analogous to Theorem~\ref{thm:conj_curves}. 

The difficult part is to compute the Beilinson--Bloch height $\hgt(Z)$ on the proper transform $Y$ of $X_0$. However, we show that $\hgt(Z)$ is, in this case, equal to the N\'eron--Tate height of a natural point in the Jacobian of a curve $C$ canonically associated to $X_0$.

\subsection{Dimensional reduction in the height computation} \label{sec:dim_reduction}

Working over a characteristic $0$ field, let $X_0 \subset \Pp^4$ be a cubic threefold with a single ordinary double point $x \in X_0$. Blow-up $x $ in $\Pp^4$ and let $Y \subset \Bl_x \Pp^4$ be the proper transform of $X_0$. Let $Q \subset Y$ be the exceptional quadric of the blow-up and let $\Lambda_1, \Lambda_2 \subset Q$ be two lines from distinct rulings of $Q$. We will consider $Z = \Lambda_1 - \Lambda_2 \in \CH_1(Y)$.

Let $C \subset \ppp \simeq \Pp(T_x \Pp^4)$ be the space of lines passing through $x$ and contained in $X_0$. Then $C$ is a smooth complete intersection curve of degree $(2,3)$, hence a canonically embedded genus four curve. It is classical that $Y \simeq \Bl_C \ppp$ and $Q \subset Y$ is the proper transform of the unique quadric containing $C$. We refer to Section~\ref{sec:proof_thm_threefold_curve} for more details. Let $\xg_i$ be the intersection of $\Lambda_i$ with $C$ and consider $\xg=\xg_1 - \xg_2$. Note that $\xg_1$ and $\xg_2$ define the two trigonal pencils on $C$.   

Suppose now that $X_0$ and the two rulings of $Q$ are defined over a number field $K$. Assume that the smooth variety $Y$ has a proper regular model over the ring of integers of $K$. We will denote by $\hgt(\xg)$ the non-normalized N\'eron--Tate height of the distinguished point $\xg$ in the Jacobian of~$C$.

\begin{theorem}\label{thm:threefold_to_curve} 
 The Beilinson--Bloch height $\hgt(Z)$ exists and equals $\hgt(\xg)$. 
\end{theorem}

The proof of this theorem is rather lengthy and we will postpone it to Section~\ref{sec:proof_thm_threefold_curve}. 
Note that a height computation for a codimension two cycle on a threefold is in general inaccessible whereas a N\'eron--Tate height computation on a curve is comparatively accessible (but, in practice, not without limitations). 

\subsection{Setting up the deformation} \label{details_calc}

We work with a nodal cubic threefold $X_0 \subset \Pp^{4}$ defined over $\Q$ having a simple node at the origin $[0:0:0:0:1]$ and given as the zero locus of a primitive polynomial $F \in \Z[x,y,z,w,u]$. 

Let $U \subset \Pp^4$ be the affine open chart where $u$ is set to $1$ and write
\begin{equation}
\label{eq:local_eqns_of_cubic}
  F|_U = f_2 + f_3
\end{equation}
with $f_d \in \Z[x,y,z,w]$  homogeneous of degree $d$. The associated genus four curve $C$ is the complete intersection $C = V(f_2,f_3) \subset \ppp$. As a deformation of $X_0$ we pick the family of hypersurfaces $X_t=V(F_t)$ with \begin{equation}
 F_t \colonequals F + t \cdot u^3,\quad t \in \Cc.
\end{equation}
We note that the generic member of this family is smooth, and that the $X_t$ degenerate simply into the threefold $X_0$.
Let $\chi = \mathrm{d}t|_0$ as usual.

\begin{example}\label{ex:explicit}
  For an explicit example, we choose
  \begin{align*}
    f_2 &=xy-zw, \\
    f_3 &= x^2w + y^2w - w^3 +  5z^3 + 2xy^2.
  \end{align*}
  Applying the method in Section~\ref{sec:how_to_compute_lmhs} we find that the limit mixed Hodge structure $\LMHS_\chi$ associated to the family $X_t$ at $t=0$ has height:
  \begin{equation} \label{eqn:ht_cubic_hs}
    \hgt(\LMHS_\chi) = 1.5338985286038602474748214314768611462429296785346\dots
  \end{equation}
  It took about 2 hours to compute this number to 195 digits of rigorous precision on a CPU with 2,3 GHz Quad-Core Intel Core i7.
\end{example}

\subsection{Computing the N\'eron--Tate height on the associated curve}

We next address the N\'eron--Tate height on the genus four curve $C$. Let $\xg_1, \xg_2$ denote the two trigonal pencils on $C$ and write $\xg=\xg_1-\xg_2$. 
We choose four degree three divisors $D_i,E_i \in |\xg_i|$, $i=1, 2$, so that the supports of the divisors $D_1,D_2,E_1,E_2$ are all disjoint. 

Following \cite{Gross1986} we have
\begin{equation} \label{eqn:ht_divisors}
\hgt(\xg) =  \langle  D_1-D_2, E_1-E_2 \rangle_{\mathrm{NT}}, 
\end{equation}
where $\langle \cdot, \cdot \rangle_{\mathrm{NT}}$ denotes the classical N\'eron height pairing, and a canonical decomposition
\begin{equation} \label{eqn:decomp}
  \langle D_1-D_2, E_1-E_2 \rangle_{\mathrm{NT}}  =  \langle D_1-D_2, E_1-E_2 \rangle _\infty + \sum_\xp \langle D_1-D_2, E_1-E_2 \rangle_\xp 
\end{equation}
of the N\'eron pairing
as a sum of an Archimedean contribution $\langle D_1-D_2, E_1-E_2 \rangle _\infty$ and non-Archimedean contributions $\langle D_1-D_2, E_1-E_2 \rangle_\xp$ indexed by the prime numbers. The sum over the prime numbers is finite, and evaluates as an element of $\Q \cdot \log |\Q^\times|$. We do not define the non-Archimedean contributions here; they are defined in a way comparable to~\eqref{eq:finite_intersection}.

\begin{example}\label{ex:explicit_2}
  We continue with our Example~\ref{ex:explicit}. Note that each of the degree three divisors $D_i,E_j$ on our genus four curve $C$ is given by a line on the quadric $V(xy-zw) \subset \ppp$. For the linear system $\xg_1$, we will use lines of the form
  \begin{equation}
    A_n \colonequals V(n x - z, nw-y), \quad n \in \Z.
  \end{equation}
  For the linear system $\xg_2$, we will use lines of the form
  \begin{equation}
    B_m \colonequals V(mx - w, mz-y), \quad m \in \Z.
  \end{equation}
  We set $D_1 \colonequals C \cdot A_{-2}$, $D_2 \colonequals C \cdot B_{-1}$, $E_1\colonequals C \cdot A_{2}$, and $E_2 \colonequals C \cdot B_{0}$.  
\end{example}

\subsection{Prime support of the non-Archimedean terms}

In principle, the non-Archimedean terms $\langle D_1-D_2, E_1-E_2 \rangle_\xp$ can be computed in a straightforward manner using a local regular model of $C$ over $\Z_\xp$. However, available computer implementations for the determination of a regular model require a smooth (weighted) \emph{plane} curve over a global field as a starting point. Nevertheless, it is easy to determine the primes $\xp$ for which $\langle D_1-D_2, E_1-E_2 \rangle_\xp$ are possibly non-zero.

Let $\cc \subset \ppp_\Z$ be the naive model of $C$ obtained by taking the zero locus of $f_2,f_3$ after clearing denominators and common factors over the integers.  Then $\langle D_1-D_2, E_1-E_2 \rangle_\xp$ can only be non-zero either if a $D_i$ intersects an $E_j$ modulo $\xp$, or if the naive integral model $\cc$ of $C$ is singular at $\xp$. The primes satisfying either of these conditions are easy to compute.

\begin{example}\label{ex:explicit_3}
  Continuing with the set-up in Examples~\ref{ex:explicit} and~\ref{ex:explicit_2} we find that the divisors may only intersect over the primes $2,5,11$ whereas the naive model $\cc$ is singular only over $2,5,31949591$.
\end{example}

\subsection{Computing the Archimedean term}

For the moment let $C$ be any smooth projective connected complex curve of positive genus~$g$. In \cite[Proposition 3.1]{vBHM2020} one finds an expression using theta functions on the Jacobian of $C$ to evaluate Archimedean N\'eron pairings of the form $\langle P_1-P_2, E_1-E_2 \rangle_\infty$, where $P_1, P_2$ are points on $C$, and where $E_1, E_2$ are divisors of degree~$g$, so that the supports of the divisors $P_1, P_2$ and $E_1, E_2$ are all disjoint. 

The formula in loc.\ cit.\ works as soon as both divisors $E_1, E_2$ are non-special. We will compute our pairings $ \langle D_1-D_2, E_1-E_2 \rangle _\infty$ by computing pairings $\langle P_1-P_2, Q_1 - Q_2 \rangle_\infty $ where $P_1, P_2, Q_1, Q_2$ are four distinct points and appealing to bilinearity. If we assume that $Q_1, Q_2$ are not Weierstrass points of $C$, then the divisors $gQ_1, gQ_2$ are not special and hence we can use \cite[Proposition 3.1]{vBHM2020} to evaluate $g \langle P_1-P_2, Q_1 - Q_2 \rangle_\infty = \langle P_1 - P_2, gQ_1 - gQ_2 \rangle_\infty$. Note that it is easy to avoid Weierstrass points  if one is interested, as we are, in divisors that can be moved in pencils.

\begin{example}\label{ex:explicit_4}
  We continue with the choices in Examples~\ref{ex:explicit} and~\ref{ex:explicit_2}. We used the small adjustment of \cite[Proposition 3.1]{vBHM2020} described above to compute
\begin{equation} \label{eqn:neron_arch}
\langle D_1-D_2, E_1-E_2 \rangle_\infty = -1.7802874760686653617706493582562792250288783953109\dots
\end{equation}
It took 6 minutes to compute this number to approximately 200 digits of precision. We stress that this number depends very much on the choice of our divisors and not just on $\xg$. 
\end{example}

\subsection{The error term}

We will now consider an easily computable set of primes that we expect should contain the support of the error term $\hgt(Z)-\hgt(\LMHS_\chi)$ in analogy to the case of curves. These are:
\begin{itemize}
  \item The bad primes of the naive model of the resolution $Y$ of $X_0$. In the case of the nodal cubic threefold, we can take the bad primes of the naive model $\cc$ of the associated genus four curve $C$. In the case of curves, these would be the primes supporting the term $\left( (\overline{p}-\overline{q})\cdot \Phi \right)_{\mathrm{fin}}$ in Theorem~\ref{thm:conj_curves}.
  \item The bad primes of the singularity of $X_0$, or equivalently, of the naive model of $Q \subset Y$. These primes are those that divide the discriminant of the quadric $f_2$ in Equation~\ref{eq:local_eqns_of_cubic}. In the case of curves, these would be the primes supporting the term $(\overline{p} \cdot \overline{q})_{\mathrm{fin}}$ in Theorem~\ref{thm:conj_curves}.
  \item The bad primes of the infinitesimal deformation $\chi$. Roughly speaking, these are the primes over which the family $X_t$ does not smooth out the singularity to first order. Over $K$, $\dd t|_0$ cuts out the quadric $Q$ in $\ppp=\Pp(T_x \Pp^4)$ and therefore can be identified with a rational multiple of $f_2$. We consider the primes supporting this multiple. In the case of curves, these would be the primes supporting the norm $\norm{\chi}$. 
\end{itemize}

We would like to emphasize that the set of (potentially) bad primes can be found by straightforward computation.

\begin{remark}
  The three collections of bad primes above are themselves not intrinsic to $X_0$ or $X/S$. Rather, they are the result of the equations, or of the naive integral models, that we are working with. Indeed, in the case of curves, all three of the non-Archimedean terms, and the primes supporting them, can be changed by changing the regular model. Of course, the total non-Archimedean term $\hgt(Z)-\hgt(\LMHS_\chi)$ must remain independent of the regular model.
\end{remark}

\begin{example}\label{ex:explicit_5}
  We continue with our running example. We would like to verify numerically that the error term $\hgt(Z) - \hgt(\LMHS_\chi)$ lies in $\Q \cdot \log |\Q^\times|$, and, more precisely, that it is supported on the bad primes of our situation. 
  
  The primes over which the naive integral model $\cc$ of the curve $C$ is singular are $2,5,31949591$ as stated in Example~\ref{ex:explicit_3}. 
  Next, there are no bad primes coming from the the singularity of $X_0$ or the given infinitesimal deformation: the discriminant $\Delta(xy-zw)$ equals $1$ and the generator $\chi$ clearly has norm~$1$. 
  
  We know already that the difference
  \begin{equation}
    \hgt(Z) - \langle D_1-D_2, E_1-E_2 \rangle_\infty = \hgt(\xg) -  \langle D_1-D_2, E_1-E_2 \rangle_\infty 
  \end{equation}
  is just the non-Archimedean contribution to the N\'eron pairing $\langle D_1-D_2, E_1-E_2 \rangle_{\mathrm{NT}} $. This difference is supported on the primes $2,5,11,31949591$ as explained in Example~\ref{ex:explicit_3}. 
  
  We expect therefore that the difference $ \langle D_1-D_2, E_1-E_2 \rangle_\infty - \hgt(\LMHS_\chi)$ is supported on the primes $2,5,11,31949591$. We observe to within numerical precision (195 digits) that
  \begin{equation}
 \langle D_1-D_2, E_1-E_2 \rangle_\infty - \hgt(\LMHS_\chi) = \log(2)-\log(5)-\log(11),
  \end{equation}
  which verifies our expectation.
  We observe that not all potentially bad primes contribute to the difference. 
\end{example}

\section{Proof of Theorem~\ref{thm:threefold_to_curve}} \label{sec:proof_thm_threefold_curve}

We start by recalling a few general results and techniques from K\"unnemann's work \cite{kunn}. 

\subsection{Higher Picard varieties} 

The following is based on \cite[Section~7]{kunn}. Fix an algebraically closed field $k$ and an integer $d \in \Z_{\ge 0}$. To each smooth projective integral variety $V$ over $k$ of dimension $d$ and each integer $p$ with $0 \le p \le d$ there is associated a \emph{higher Picard variety} $\Pic^p V$. This is an abelian variety, which comes together with an Abel--Jacobi mapping $\theta^p \colon A^p(V) \to \Pic^p(V)(k)$ which is universal for so-called \emph{Picard homomorphisms} from $A^p(V)$ to groups $B(k)$ where $B$ is any abelian variety over $k$. Here $A^p(V)$ denotes the group of codimension-$p$ cycles on $V$ algebraically equivalent to zero, modulo rational equivalence. 

For all smooth projective integral varieties $V, W$ over $k$, with $\dim V=d$, $\dim W=e$, and all correspondences $\alpha \in \CH^{e-p+q}(W \times_k V)$, we have induced homomorphisms $\Pic(\alpha) \colon \Pic^p(W) \to \Pic^q(V)$ of associated Picard varieties, enjoying natural compatibilities with the Abel--Jacobi mappings. 

Furthermore, we have natural isogenies $\lambda_V^{d+1-p} \colon \Pic^{d+1-p}(V) \to \Pic^p(V)^\lor$ as well as uniquely determined positive integers $k_V^p$ satisfying the following property: when $\alpha \in \CH^{p+q}(W \times_k V)$ is a correspondence, we have an identity
\begin{equation} \label{k-identity}
[k_V^p] \circ \lambda_W^{q+1} \circ \Pic({}^t \alpha) = [k_W^{e-q}] \circ \Pic(\alpha)^\lor \circ \lambda_V^{d+1-p}
\end{equation}
of homomorphisms from $\Pic^{d+1-p}(V)$ to $\Pic^{e-q}(W)^\lor$, cf.\ \cite[Equation 22]{kunn}. When $C$ is a smooth projective integral curve, we have $k_C^1=1$ and $\lambda_C^1$ coincides with \emph{minus} the canonical principal polarization of $\Pic^1 C$.

\subsection{Connection with the N\'eron--Tate pairing} \label{sec:kunn_result}

The following is based on \cite[Section~8]{kunn}.
Let $K$ be a number field, and let $\overline{K}$ be an algebraic closure of $K$. When $B$ is an abelian variety over $\overline{K}$, we denote by $\langle \cdot,\cdot \rangle_{B}$ the  classical N\'eron--Tate height pairing on $B(\overline{K}) \times B^\lor(\overline{K})$ induced by the Poincar\'e bundle \cite{Neron1965}.  

The height pairing in the sense of Beilinson and Bloch \cite{Bei1987, Bloch1984}  \emph{of cycles algebraically equivalent to zero} can be calculated as a certain N\'eron--Tate pairing, using appropriate higher Picard varieties and higher Abel--Jacobi mappings.
\begin{theorem} \label{kunnemann_main} (=\cite[Theorem~8.2]{kunn})  Let $V_K$ be a smooth projective geometrically integral variety over the number field $K$ of dimension~$d$. Let $Z_1 \in A^p(V_K)$ and $Z_2 \in A^{d+1-p}(V_K)$ be cycles algebraically equivalent to zero on $V_K$. Assume that $V_K$ has a proper regular model over the ring of integers of $K$. Then the Beilinson--Bloch pairing $\langle Z_1,Z_2 \rangle$ between $Z_1, Z_2$ exists. Moreover, the equality
\begin{equation}
  \langle Z_1, Z_2 \rangle = \frac{1}{k_V^p} [K:\Q]  \langle \theta_V^p(Z_1), \lambda_V^{d+1-p} \circ \theta_V^{d+1-p}(Z_2) \rangle_{\Pic^p(V_{\overline{K}})} 
\end{equation}
holds in $\R$.
\end{theorem}
The following standard property of N\'eron--Tate height pairings will be used later.
\begin{proposition} \label{lem:NT_pairing_duality} Let $f \colon A \to B$ be a homomorphism of abelian varieties over $\overline{K}$. Let $f^\lor \colon B^\lor \to A^\lor$ be the associated dual morphism. Then we have $\langle f(\alpha), \beta \rangle_B = \langle \alpha, f^\lor(\beta) \rangle_A$ for $\alpha \in A(\overline{K})$, $\beta \in B^\lor(\overline{K})$.
\end{proposition}
\begin{proof} This follows from the compatibility of the rigidified metrized line bundles $(\id,f^\lor)^*\calP_A$ and $(f,\id)^*\calP_B$ on $A \times B^\lor$, where $\calP_A$ resp.\ $\calP_B$ are the rigidified and canonically metrized Poincar\'e bundles on $A \times A^\lor$ resp.\ $B \times B^\lor$. See \cite[Equation 3]{kunn}.
\end{proof}
\begin{definition} \label{def:bb_height}
When $V$ has dimension $n=2p-1$ and $Z$ is a homologically trivial codimension-$p$ cycle on $V$, the
Beilinson--Bloch height $\hgt(Z)$ is given as $(-1)^p \langle Z,Z \rangle$.
\end{definition}
Beilinson's ``arithmetic standard conjectures of Hodge Index type'', see e.g.\ \cite[Conjecture 5.2]{kunn}, predict that $\hgt(Z) \ge 0$. This explains the insertion of the sign $(-1)^p$ in front of the self-pairing in order to obtain the height.

\subsection{Weil intermediate Jacobians}

The following is based on \cite[Sections~10 and 11]{kunn}. Let $V$ be a smooth projective complex variety. Let $J^p(V)=\H^{2p-1}(V,\R)/\H^{2p-1}(V,\Z)$ be the $p$-th \emph{Weil intermediate Jacobian} of $V$. This is a complex abelian variety, receiving an Abel--Jacobi map $\Phi^p \colon \CH^p(X)^0 \to J^p(V)$, with $\CH^p(X)^0$  the group of codimension-$p$ cycles modulo rational equivalence on $V$ that are homologically trivial.  

Let $q_V^p \colon \H^{2p-1}(V,\Z) \times \H^{2d-2p+1}(V,\Z) \to \Z$ be the pairing given by Poincar\'e duality, and take its real linear extension to real-valued cohomology. Then the pairing $\H^{2p-1}(V,\R) \times \H^{2d-2p+1}(V,\R) \to \Cc$ given by $(\alpha,\beta) \mapsto q_V^p(\ii \alpha,\beta) + \ii q_V^p(\alpha,\beta)$ is non-degenerate and induces an isomorphism  $\mathrm{pd} \colon J^{p}(V) \to J^{d+1-p}(V)^\lor$ of complex abelian varieties. 

The following result is implicit in the proof of \cite[Theorem~9.1]{kunn}.
\begin{proposition} \label{prop:lambda_poinc_compatible}
There is a natural morphism of abelian varieties $j_V^p \colon \Pic^p(V) \to J^p(V)$ with kernel contained in the $k_V^p$-torsion and satisfying the following property. Let $\mathrm{pd} \colon J^p(V) \to J^{d+1-p}(V)^\lor$ be the isomorphism induced by Poincar\'e duality. Then $\lambda_V^p$ and $\mathrm{pd}$ are compatible in the sense that the equality $\lambda_V^p = j_V^{d+1-p, \lor} \circ \mathrm{pd} \circ j_V^p$ holds as maps from $\Pic^p(V)$ to $\Pic^{d+1-p}(V)^\lor$.
\end{proposition}

\subsection{Geometry of nodal cubic threefolds} \label{sec:experiment_outline} \label{sec:genus_four}

We elaborate a bit more on the classical geometry of nodal cubic threefolds. For a more detailed discussion of the facts we mention below we refer to the lecture notes \cite[Section~5.1]{huybrechts}.

Let $K$ be a field of characteristic zero. Let $X_0 \subset \Pp^4$ be a cubic threefold with a single ordinary double point $x$. Let $Y$ denote the proper transform of $X_0$ under the blow-up of $\Pp^4$ at $x$, let $Q \subset Y$ denote the exceptional quadric, let $\Lambda_1, \Lambda_2$ denote two lines from distinct rulings of $Q$, and set  $Z = \Lambda_1 - \Lambda_2$. Let $C \subset \Pp^3$ be the moduli of lines that pass through $x \in \Pp^4$ and are contained in $X_0$. Then $C$ is a canonically embedded genus four curve.

We have a canonical identification $Y \simeq \Bl_C \ppp$ obtained from projecting $X_0$ from its node. Let $p \colon Y \to \ppp$ be the resulting projection map,  let $E =p^{-1}C \subset Y$ be the exceptional divisor of $p$, and let  $j \colon E \to Y$ be the inclusion map.

The operation $j_* \circ p^*$ gives an isomorphism $A^1(C) \isom A^2(Y)$ of groups of cycles algebraically equivalent to zero modulo rational equivalence.  We have $C \simeq E \cap Q$. The two rulings $\Lambda_1, \Lambda_2$  of $Q$ induce by restriction to $C$ the two canonical trigonal pencils on $C$ which we call $\xg_1$ resp.\ $\xg_2$. The difference $\xg = \xg_1 - \xg_2$ gives a non-trivial element in $A^1(C)$. 
\begin{proposition} \label{prop:identify_pencil_Z} The isomorphism $j_* \circ p^* \colon A^1(C) \isom A^2(Y)$ described above sends the class of $\xg$ into the class of $Z$. In particular, the cycle $Z$ is algebraically equivalent to zero.
\end{proposition}

When $K$ is a subfield of $\Cc$ the operation $j_* \circ p^*$ also induces an isomorphism $\H^1(C,\Z) \isom \H^3(Y,\Z)$. This isomorphism carries the intersection pairing $q_C^1$ on $\H^1(C,\Z)$ into \emph{minus} the intersection pairing $q_Y^2$ on $\H^3(Y,\Z)$. This result may be rephrased as saying that the operation $j_* \circ p^*$ induces a Hodge isometry $\H^1(C,\Z) \isom \H^3(Y,\Z(1))$. We deduce the following proposition.
\begin{proposition} \label{prop:two_poinc_dual} (a) The operation $j_* \circ p^*$ gives an isomorphism $J^1(C) \isom J^2(Y)$ of Weil intermediate Jacobians. (b) Via this isomorphism and its dual, the isomorphism induced by Poincar\'e duality $J^1(C) \isom J^1(C)^\lor$ is identified with minus the isomorphism induced by Poincar\'e duality $J^2(Y) \isom J^2(Y)^\lor$.
\end{proposition}

\subsection{The Picard variety of the proper transform}

We continue with the notation of the previous section, but taking $K$ to be a number field from now on. 
Note that we may view $j_* \circ p^*$ as a correspondence $\alpha \in \CH^2(C \times Y)$. By applying the formalism as elaborated in Section~\ref{sec:kunn_result}  we obtain from this an isomorphism of Picard varieties $\Pic(\alpha) \colon \Pic^1(C) \isom \Pic^2(Y)$  compatible with the Abel--Jacobi mappings $\theta_C^1 \colon A^1(C) \to \Pic^1(C)$ and $\theta_Y^2 \colon A^2(Y) \to \Pic^2(Y)$ and the isomorphism $A^1(C) \isom A^2(Y)$. 

\begin{lemma} \label{lem:identify_in_Pic} We have an identity
 $\Pic(\alpha)(\theta_C^1(\xg)) = \theta_Y^2(Z)$ of elements of $\Pic^2(Y)$.
\end{lemma}
\begin{proof} This follows directly from Proposition~\ref{prop:identify_pencil_Z} and the compatibility of the maps discussed above.
\end{proof}

\begin{lemma} \label{lem:minus-sign} Under the isomorphism $\Pic(\alpha) \colon \Pic^1(C) \isom \Pic^2(Y)$ and its dual, the isomorphism $\lambda_Y^2$ is identified with $\mathrm{mult}( -1) \circ \lambda^1_C$. In particular $\lambda^2_Y$ is a principal polarization of $\Pic^2(Y)$.
\end{lemma}
\begin{proof} Let $\sigma \colon K \hookrightarrow \Cc$ be a complex embedding of $K$. 
Let $J^1(C_\sigma)$ and $J^2(Y_\sigma)$ denote the first resp.\ second Weil intermediate Jacobian associated to $C_\sigma$ resp.\ $Y_\sigma$. Each of the Abel--Jacobi mappings $A^1(C_\sigma) \to \Pic^1(C_\sigma)$, $A^1(C_\sigma) \to J^1(C_\sigma)$, $A^2(Y_\sigma) \to \Pic^2(Y_\sigma)$ and $A^2(Y_\sigma) \to J^2(Y_\sigma)$ is an isomorphism, and the maps $j_{C_\sigma}^1 \colon \Pic^1(C_\sigma) \to J^1(C_\sigma)$ and 
$j_{Y_\sigma}^2 \colon \Pic^2(Y_\sigma) \to J^2(Y_\sigma)$ are the unique isomorphisms compatible with these Abel--Jacobi isomorphisms. 

Let $\mathrm{pd}_{C,\sigma} \colon J^1(C_\sigma) \isom J^1(C_\sigma)^\lor$ be the isomorphism induced by Poincar\'e duality. By Proposition~\ref{prop:lambda_poinc_compatible} we have that under the isomorphism $j_{C_\sigma}^1 \colon \Pic^1(C_\sigma) \isom J^1(C_\sigma)$ and its dual, the isomorphism $\lambda_{C,\sigma}^1 \colon \Pic^1(C_\sigma) \isom \Pic^1(C_\sigma)^\lor$ is identified with $\mathrm{pd}_{C_\sigma}$. Similarly, 
under the isomorphism $j_{Y_\sigma}^2 \colon \Pic^2(Y_\sigma) \isom J^2(Y_\sigma)$ and its dual, the isomorphism $\lambda_{Y,\sigma}^2 \colon \Pic^2(Y_\sigma) \isom \Pic^2(Y_\sigma)^\lor$ is identified with the  isomorphism $\mathrm{pd}_{Y,\sigma} \colon J^2(Y_\sigma) \isom J^2(Y_\sigma)^\lor$ induced by Poincar\'e duality. 

By Proposition~\ref{prop:two_poinc_dual}(a) the operation $j_* \circ p^*$ gives rise to an isomorphism $J^1(C_\sigma) \isom J^2(Y_\sigma)$. This isomorphism is compatible with the isomorphisms $j_{C_\sigma}^1 \colon \Pic^1(C_\sigma) \isom J^1(C_\sigma)$ and $j_{Y_\sigma}^2 \colon \Pic^2(Y_\sigma) \isom J^2(Y_\sigma)$ and $\Pic(\alpha) \colon \Pic^1(C) \isom \Pic^2(Y)$.  By Proposition~\ref{prop:two_poinc_dual}(b) we have that via the isomorphism $J^1(C_\sigma) \isom J^2(Y_\sigma)$ and its dual,  the isomorphism $\mathrm{pd}_{C,\sigma}$  is identified with $\mathrm{mult}(-1) \circ \mathrm{pd}_{Y,\sigma}$. 

The lemma follows upon combining these observations.
\end{proof}

\begin{lemma} \label{lem:the_integer}
We have $k_Y^2=1$.
\end{lemma}
\begin{proof} 
From Equation \ref{k-identity}  we may deduce the identity
\begin{equation} \label{k-identity-special}
[k_C^1] \circ \lambda_Y^2 \circ \Pic( \alpha) = [k_Y^2] \circ \Pic({}^t \alpha)^\lor \circ \lambda^1_C
\end{equation}
of homomorphisms from $\Pic^1(C) $ to $\Pic^2(Y)^\lor$. Each of the homomorphisms $\Pic( \alpha)$, $ \Pic({}^t \alpha)^\lor$ and $\lambda^1_C$ is an isomorphism. By Lemma~\ref{lem:minus-sign} we have that $\lambda_Y^2$ is an isomorphism. We have noted above that $k_C^1=1$. As our Picard varieties are positive dimensional and $k_Y^2$ is a positive integer we conclude that $k_Y^2=1$ as well.
\end{proof}

\subsection{Finishing the proof}

We have now all ingredients available to finish the proof of Theorem~\ref{thm:threefold_to_curve}.

\begin{proof}[Proof of Theorem~\ref{thm:threefold_to_curve}] We compute
 \begin{align*}
 \hgt(Z) & =  \langle Z, Z \rangle && \textrm{by Definition~\ref{def:bb_height}} \\
    & = [K:\Q] \langle \theta^2_Y(Z), \lambda^2_Y \circ \theta^2_Y(Z) \rangle _{\Pic^2(Y_{\overline{K}})} && \textrm{by Theorem~\ref{kunnemann_main}} \\
    & = [K:\Q]\langle \Pic(\alpha)(\theta^1_C(\xg)),  \lambda^2_Y \circ \Pic(\alpha)(\theta^1_C(\xg)) \rangle _{\Pic^2(Y_{\overline{K}})} && \textrm{by Lemma~\ref{lem:identify_in_Pic}} \\
    & = [K:\Q]\langle \theta^1_C(\xg), \Pic(\alpha)^\lor \circ \lambda_Y^2 \circ \Pic(\alpha)(\theta_C^1(\xg)) \rangle _{\Pic^1(C_{\overline{K}})} && \textrm{by Proposition~\ref{lem:NT_pairing_duality}} \\
    & = - [K:\Q]\langle \theta_C^1(\xg), \lambda_C ^1(\theta^1_C(\xg)) \rangle _{\Pic^1(C_{\overline{K}})} && \textrm{by Lemma~\ref{lem:minus-sign}} \\
    & = - [K:\Q]\frac{1}{k_Y^2}\langle \theta_C^1(\xg), \lambda_C ^1(\theta^1_C(\xg)) \rangle _{\Pic^1(C_{\overline{K}})} && \textrm{by Lemma~\ref{lem:the_integer}} \\
    & = -   \langle \xg, \xg \rangle && \textrm{by Theorem~\ref{kunnemann_main}} \\
    & = \hgt(\xg) && \textrm{by Definition~\ref{def:bb_height}}
    \end{align*}
This completes the proof of Theorem~\ref{thm:threefold_to_curve}.
\end{proof}

\begin{remark}  \label{rem:not_alg_equiv} 
  Ceresa and Collino have shown in \cite{CeresaCollino1983} that when $X_0 \subset \Pp^4$ is a general threefold \emph{of degree at least~$5$}, having a single ordinary node, the cycle $Z$ is \emph{not} algebraically trivial. This indicates that in general there might be no direct connection available that links $\hgt(Z)$ with a certain N\'eron--Tate height pairing. Although the strategy implicit in Conjecture~\ref{conj:bloch} will continue to hold, we no longer have the means to check the result. 
\end{remark}

\bibliographystyle{plain}
\bibliography{heights}

\end{document}